\documentclass[12pt]{amsart}
\usepackage{amsmath,amsthm,amsfonts,amssymb,mathrsfs}
\date{\today}

\usepackage{color}

\usepackage{hyperref}

\newtheorem{theorem}{Theorem}

\newtheorem{corollary}[theorem]{Corollary}
\newtheorem{lemma}[theorem]{Lemma}
\theoremstyle{definition}
\newtheorem{example}[theorem]{Example}
\newtheorem{remark}[theorem]{Remark}

\begin{document}

\title[ON THE DICHOTOMY OF A LOCALLY COMPACT SEMITOPOLOGICAL MONOID OF...]{ON THE DICHOTOMY OF A LOCALLY COMPACT SEMITOPOLOGICAL MONOID OF ORDER ISOMORPHISMS BETWEEN PRINCIPAL FILTERS OF $\mathbb{N}^n$ WITH ADJOINED ZERO}

\author[Taras~Mokrytskyi]{Taras~Mokrytskyi}
\address{Faculty of Mathematics, National University of Lviv,
Universytetska 1, Lviv, 79000, Ukraine}
\email{tmokrytskyi@gmail.com}



\begin{abstract}
Let $n$ be any positive integer and $\mathscr{I\!\!P\!F}(\mathbb{N}^n)$ be the semigroup of all order isomorphisms between principal filters of the $n$-th power of the set of positive integers $\mathbb{N}$ with the product order.  We prove that a Hausdorff locally compact semitopological semigroup  ${\mathscr{I\!\!P\!F}(\mathbb{N}^n)}$ with an adjoined zero is either compact or discrete.

\end{abstract}

\maketitle

Further we shall follow the terminology of \cite{Carruth-Hildebrant-Koch-1983-1986, Clifford-Preston-1961-1967, Engelking-1989, Ruppert-1984}. In this paper we shall denote the set of positive integers by $\mathbb{N}$, the set of non-negative integers by $\mathbb{N}_0$, a semigroup $S$ with the an adjoined by $S^0$ (cf. \cite{Clifford-Preston-1961-1967}), the symmetric group of degree $n$ by $\mathscr{S}_n$, i.e., $\mathscr{S}_n$ is the group of all permutations of an $n$-element set. All topological spaces, considered in this paper, are assumed to be Hausdorff.

\smallskip

A semigroup $S$ is called \emph{inverse} if for every $x \in S$ there exists a unique $y \in S$ such that $xyx = x$ and $yxy = y$. Later such an element $y$ will be denoted by $x^{-1}$ and will be called the \emph{inverse} of $x$. A map $\textrm{inv}\colon S \rightarrow S$ which assigns to every $s \in S$ its inverse is called \emph{inversion}.

\smallskip

If $Y$ is a subspace of a topological space $X$ and $A \subset Y$, then by $\textrm{cl}_Y(A)$ we denote the topological closure of $A$ in $Y$.

\smallskip

A \emph{semitopological (topological) semigroup} is a topological space with separately continuous (jointly continuous) semigroup operation. An inverse topological semigroup with continuous inversion is called \emph{a topological inverse semigroup}.

\smallskip

We recall that a topological space X is \emph{locally compact} if every point $x$ of $X$ has an open neighbourhood $U(x)$ with the compact closure
$\textrm{cl}_X(U(x))$.

\smallskip

The \emph{bicyclic semigroup} (or the \emph{bicyclic monoid}) $\mathscr{C}(p,q)$ is the semigroup with the identity 1 generated by elements $p$ and $q$ and the relation $pq = 1$.

\smallskip

The bicyclic semigroup plays an important role in algebraic theory of semigroups and in the theory of topological semigroups. For instance a well-known Andersen's result~\cite{Andersen-1952} states that a ($0$--)simple semigroup with an idempotent is completely ($0$--)simple if and only if it does not contain an isomorphic copy of the bicyclic semigroup. The bicyclic monoid admits only the discrete semigroup topology. Bertman and  West in \cite{Bertman-West-1976} extended this result for the case of semitopological semigroups. Stable and $\Gamma$-compact topological semigroups do not contain the bicyclic monoid~\cite{Anderson-Hunter-Koch-1965, Hildebrant-Koch-1986}. The problem of an embedding of the bicyclic monoid into compact-like topological semigroups was studied in \cite{Banakh-Dimitrova-Gutik-2009, Banakh-Dimitrova-Gutik-2010, Gutik-Repovs-2007}.

\smallskip

For an arbitrary positive integer $n$ by $\left(\mathbb{N}^n,\leqslant\right)$ we denote the $n$-th power of the set of positive integers $\mathbb{N}$ with the product order:
\begin{equation*}
  \left(x_1,\ldots,x_n\right)\leqslant\left(y_1,\ldots,y_n\right) \qquad \hbox{if and only if} \qquad x_i\leq y_i \quad \hbox{for all} \quad i=1,\ldots,n.
\end{equation*}
It is obvious that the set of all order isomorphisms between principal filters of the poset $\left(\mathbb{N}^n,\leqslant\right)$ with the operation of composition of partial maps form a semigroup. This semigroup will be denoted by $\mathscr{I\!\!P\!F}(\mathbb{N}^n)$. The semigroup $\mathscr{I\!\!P\!F}(\mathbb{N}^n)$ is a generalization of the bicyclic semigroup ${\mathscr{C}}(p,q)$. Hence it is natural to ask: \emph{what algebraic and topological properties of the semigroup $\mathscr{I\!\!P\!F}(\mathbb{N}^n)$ are similar to these of the bicyclic monoid?} The structure of the semigroup $\mathscr{I\!\!P\!F}(\mathbb{N}^n)$ is studied in \cite{Gutik-Mokrytskyi-2019}. There was shown that $\mathscr{I\!\!P\!F}(\mathbb{N}^n)$ is a bisimple, $E$-unitary, $F$-inverse monoid, described Green's relations on $\mathscr{I\!\!P\!F}(\mathbb{N}^n)$ and its maximal subgroups. It was proved that $\mathscr{I\!\!P\!F}(\mathbb{N}^n)$ is isomorphic to the semidirect product of the direct $n$-th power of the bicyclic monoid ${\mathscr{C}}^n(p,q)$ by the group of permutation  $\mathscr{S}_n$, every non-identity congruence on  $\mathscr{I\!\!P\!F}(\mathbb{N}^n)$ is group and was described the least group congruence on $\mathscr{I\!\!P\!F}(\mathbb{N}^n)$. It was shown that every shift-continuous topology on $\mathscr{I\!\!P\!F}(\mathbb{N}^n)$ is discrete and discussed embedding of the semigroup $\mathscr{I\!\!P\!F}(\mathbb{N}^n)$ into compact-like topological semigroups.

\smallskip

A dichotomy for the bicyclic monoid with an adjoined zero $\mathscr{C}^0={\mathscr{C}}(p,q)\sqcup\{0\}$ was proved in \cite{Gutik-2015}: \emph{every locally compact semitopological bicyclic monoid $\mathscr{C}^0$ with an adjoined zero is either compact or discrete}.
The above dichotomy was extended by Bardyla in \cite{Bardyla-2016} to locally compact $\lambda$-polycyclic semitopological monoids, in \cite{Bardyla-2018a} to locally compact semitopological graph inverse semigroups in \cite{Gutik-Maksymyk-2016} to locally compact semitopological interassociates of the bicyclic monoid with an adjoined zero and are extended in \cite{Gutik-2018} to locally compact semitopological $0$-bisimple inverse $\omega$ semigroups with compact maximal subgroups.

\smallskip

The main purpose of this paper is to obtain counterparts of the above results to locally compact semitopological monoid $\mathscr{I\!\!P\!F}(\mathbb{N}^n)$.

\smallskip

By $\mathscr{I\!\!P\!F}(\mathbb{N}^n)^0$ we denote the monoid $\mathscr{I\!\!P\!F}(\mathbb{N}^n)$ with an adjoined zero.

\begin{lemma}\label{lemma-1.1}
Let $({\mathscr{I\!\!P\!F}(\mathbb{N}^n)}^0,\tau)$ be a locally compact non-discrete semitopological semigroup. Then:
\begin{itemize}
  \item[$(1)$] for every open neighbourhood $U(0)$ of the zero in $({\mathscr{I\!\!P\!F}(\mathbb{N}^n)}^0,\tau)$ there exists an open compact neighbourhood $V(0)$ of the zero in $({\mathscr{I\!\!P\!F}(\mathbb{N}^n)}^0,\tau)$ such that ${V(0)\subset U(0)}$;
  \item[$(2)$] for every open neighbourhood $U(0)$ of the zero in $({\mathscr{I\!\!P\!F}(\mathbb{N}^n)}^0,\tau)$ and every open compact neighbourhood $V(0)$ of the zero in $({\mathscr{I\!\!P\!F}(\mathbb{N}^n)}^0,\tau)$ the set $V(0) \cap U(0)$ is compact and open, and the set $V(0)\setminus U(0)$ is finite.
\end{itemize}
\end{lemma}

\begin{proof}[\textsl{Proof}]
$(1)$ Let $U(0)$ be an arbitrary open neighbourhood of the zero in $({\mathscr{I\!\!P\!F}(\mathbb{N}^n)}^0,\tau)$. By Theorem 3.3.1 from [10] the
space $({\mathscr{I\!\!P\!F}(\mathbb{N}^n)}^0,\tau)$ is regular. Since it is locally compact, there exists an open neighbourhood $V(0) \subseteq U(0)$
of the zero in $({\mathscr{I\!\!P\!F}(\mathbb{N}^n)}^0,\tau)$ such that ${cl_{{\mathscr{I\!\!P\!F}(\mathbb{N}^n)}^0} (V(0)) \subseteq U(0)}$. Since all non-zero elements of the semigroup ${\mathscr{I\!\!P\!F}(\mathbb{N}^n)}^0$ are isolated points in $({\mathscr{I\!\!P\!F}(\mathbb{N}^n)}^0,\tau)$, \break ${\textrm{cl}_{{\mathscr{I\!\!P\!F}(\mathbb{N}^n)}^0}(V(0))=V(0)}$, and hence our assertion holds.

\smallskip

$(2)$ Let $V(0)$ be an arbitrary compact open neighbourhood of the zero in $({\mathscr{I\!\!P\!F}(\mathbb{N}^n)}^0,\tau)$. Then for an arbitrary open neighbourhood $U(0)$ of the zero in $({\mathscr{I\!\!P\!F}(\mathbb{N}^n)}^0,\tau)$ the family
\begin{equation*}
\mathscr{U}=\{U(0)\}\cup\left\{\{x\} \colon  x\in V(0)\setminus U(0)\right\}
\end{equation*}
is an open cover of $V(0)$. Since the family $\mathscr{U}$ is disjoint, it is finite. So the set $V(0)\setminus U(0)$ is finite and hence the set $V(0)\cap U(0)$ is compact.
\end{proof}



\begin{remark}\label{remark-2.17}
On the bicyclic semigroup ${\mathscr{C}}(p,q)$ the semigroup operation is determined in the following way:
\begin{equation*}
  p^iq^j\cdot p^kq^l=
\left\{
  \begin{array}{ll}
    p^iq^{j-k+l}, & \hbox{if~} j>k;\\
    p^iq^l,       & \hbox{if~} j=k;\\
    p^{i-j+k}q^l, & \hbox{if~} j<k,
  \end{array}
\right.
\end{equation*}
which is equivalent to the following multiplication:
\begin{equation*}
  p^iq^j\cdot p^kq^l=p^{i+\max\{j,k\}-j}q^{l+\max\{j,k\}-k}.
\end{equation*}
The above implies that the bicyclic semigroup ${\mathscr{C}}(p,q)$ is isomorphic to the semigroup $(\mathbb{N}_0\times \mathbb{N}_0,*)$ which is defined on the square $\mathbb{N}_0\times \mathbb{N}_0$ of the set of non-negative integers with the following multiplication:
\begin{equation}\label{eq-2.1}
  (i,j)*(k,l)=(i+\max\{j,k\}-j,l+\max\{j,k\}-k).
\end{equation}
\end{remark}

We note that the semigroup $(\mathbb{N}_0\times \mathbb{N}_0,*)$ is isomorphic to the semigroup $(\mathbb{N}\times \mathbb{N},*)$ which is defined on the square $\mathbb{N}\times \mathbb{N}$ of the set of all positive integers with the same operation $*$. It is obvious that the map $f\colon \mathbb{N}_0\times \mathbb{N}_0 \rightarrow \mathbb{N}\times \mathbb{N}$, $(i,j) \mapsto (i+1,j+1)$ is an isomorphism between semigroups $(\mathbb{N}_0\times \mathbb{N}_0,*)$ and $(\mathbb{N}\times \mathbb{N},*)$.

\smallskip

In this paper we will use the semigroup $(\mathbb{N}\times \mathbb{N},*)$ as a representation of the bicyclic semigroup ${\mathscr{C}}(p,q)$.

\smallskip

For an arbitrary positive integer $n$ by ${\mathscr{C}}(p,q)^n$ we shall denote the $n$-th direct power of $(\mathbb{N}\times \mathbb{N},*)$, i.e., ${\mathscr{C}}(p,q)^n$ is the $n$-th power of $\mathbb{N}\times \mathbb{N}$ with the point-wise semigroup operation defined by formula~\eqref{eq-2.1}. Also, by $[\mathbf{x},\mathbf{y}]$ we denote the ordered collection $\left((x_1,y_1),\dots,(x_n,y_n)\right)$ of ${\mathscr{C}}(p,q)^n$, where $\mathbf{x}=(x_1,\ldots,x_n)$ and $\mathbf{y}=(y_1,\ldots,y_n)$, and for arbitrary permutation $\sigma\colon\left\{1,\ldots,n\right\}\to\left\{1,\ldots,n\right\}$  we put
\begin{equation*}\label{eq-2.2}
(\mathbf{x})\sigma =\left(x_{(1)\sigma^{-1}},\ldots,x_{(n)\sigma^{-1}}\right).
\end{equation*}

We recall (cf. \cite{Gutik-Mokrytskyi-2019}) that the semigroup $\mathscr{I\!\!P\!F}(\mathbb{N}^n)$ is isomorphic to the semidirect product  \break $\mathscr{S}_n\ltimes_{}{\mathscr{C}}(p,q)^n$ and hence according the above arguments we can consider the semigroup $\mathscr{I\!\!P\!F}(\mathbb{N}^n)$ as the set $\mathscr{S}_n \times (\mathbb{N}\times \mathbb{N})^n$  with the following semigroup operation
\begin{equation*}
\begin{split}
  (\alpha, [\textbf{x},\textbf{y}])&\cdot(\beta, [\textbf{u},\textbf{v}]) =(\alpha\circ\beta, [(\textbf{x})\beta,(\textbf{y})\beta]\ast[\textbf{u},\textbf{v}])= \\
   = & \; (\alpha\circ\beta, [(\textbf{x})\beta + \max\{(\textbf{y})\beta,\textbf{u}\} - (\textbf{y})\beta, \textbf{v} + \max\{(\textbf{y})\beta,\textbf{u}\} - \textbf{u} ])
\end{split}
\end{equation*}

For any permutation $\sigma \in \mathscr{S}_n$ of an $n$-element set and for any ordered tuple ${\textbf{a}=(a_1,\ldots,a_n)\in \mathbb{N}^n}$ we put
\begin{equation*}
L_\sigma^\textbf{a}=\left\{(\sigma, [\textbf{a},\textbf{x}]) \in \mathscr{I\!\!P\!F}(\mathbb{N}^n) \colon \textbf{x}\in \mathbb{N}^n \right\}.
\end{equation*}

For any integer $i\in \{1,\ldots,n\}$ define an element $\textbf{2}_i$ as an element of $\mathbb{N}^n$ with the property that only $i$-th coordinate of $\textbf{2}_i$ is equal to $2$ and all other coordinate are equal to $1$, i.e. $\textbf{2}_i=(1,\ldots,\underbrace{2}_{i}, \ldots, 1)$.

\begin{lemma}\label{lemma-1.2}
Let $({\mathscr{I\!\!P\!F}(\mathbb{N}^n)}^0,\tau)$ be a locally compact non-discrete semitopological semigroup. Then for any neighborhood $U(0)$ of the zero $0$ and for any permutation $\sigma \in \mathscr{S}_n$ there exists $\textbf{a}\in \mathbb{N}^n$ such that the set $L_\sigma^\textbf{a} \cap U(0)$ is infinite.
\end{lemma}

\begin{proof}[\textsl{Proof}]
Suppose to the contrary that there exists neighborhood $U(0)$ of the zero $0$ and permutation $\sigma \in \mathscr{S}_n$ such that for any $\textbf{a}\in \mathbb{N}^n$ the set $L_\sigma^\textbf{a} \cap U(0)$ is finite. Then Lemma~\ref{lemma-1.1}(1) and the separate continuity of the semigroup operation in $({\mathscr{I\!\!P\!F}(\mathbb{N}^n)}^0,\tau)$ imply that there exists an open compact neighbourhood $V(0)$ of the zero $0$ in $({\mathscr{I\!\!P\!F}(\mathbb{N}^n)}^0,\tau)$ such that $V(0) \cdot (1, [\textbf{1}, \textbf{2}_{1}])\subset U(0)$.

\smallskip

Since for any fixed element $\textbf{a}\in \mathbb{N}^n$ the set $L_\sigma^\textbf{a} \cap U(0)$ is finite, there exists an element $$m_\textbf{a}=(\sigma, [\textbf{a}, (x_1,\ldots,x_n)]) \in L_\sigma^\textbf{a} \cap U(0)$$
with property that
\begin{equation}\label{equation-1}
U(0) \not\owns (\sigma, [\textbf{a}, (x_1+1,\ldots,x_n)]) = m_\textbf{a}\cdot (1, [\textbf{1}, \textbf{2}_{1}]).
\end{equation}

Consider the set $M=\left\{m_\textbf{a} \colon \textbf{a}\in \mathbb{N}^n \right\}$. Then property (\ref{equation-1}) implies that $M\cap V(0)=\varnothing$. Thus ${U(0)\setminus V(0) \supset M}$ which contradicts Lemma ~\ref{lemma-1.1}(2) because the set $M$ is infinite.
\end{proof}

\begin{lemma}\label{lemma-1.3}
Let $n$ be a positive integer, $A$ and $B$ be infinite subsets of $\mathbb{N}^n$ such that $A\sqcup B=\mathbb{N}^n$ and $A\cap B=\emptyset$. Then there exist an infinite subset $C\subset A$ and a positive integer $k\in \{1,\ldots,n\}$ such that at least one of this two sets $(C)g_k$ and $(C)g^{-1}_k$ 
is a subset of $B$, where $g_k$ is the map from $\mathbb{N}^n$ to $\mathbb{N}^n$ defined in the following way: ${(x_1,\ldots,x_n)g_k=(x_1,\ldots,x_k+1,\ldots,x_n)}$.
\end{lemma}

\begin{proof}[\textsl{Proof}]

If $n=1$ consider the the set $C = \{a \in A \colon a+1 \in B \}$, $C$ is infinite and $(C)g_1 \subset B$.

\smallskip

Let $n \geq 2$.
An ordered tuple $p=\left(\textbf{p}^1,\textbf{p}^2,\ldots,\textbf{p}^{r-1},\textbf{p}^r\right) \in (\mathbb{N}^n)^r$ of elements of $\mathbb{N}^n$ is called a \emph{path from point $\textbf{a}$ to point $\textbf{b}$} if $\textbf{p}^1~=~\textbf{a}$, $\textbf{p}^k=\textbf{b}$ and for any index $i\in \{2,\ldots,k\}$ there exist some $m_i\in \{1,\ldots,n\}$ such that $(\textbf{p}^{i-1})g_{m_i}=\textbf{p}^{i}$ or $(\textbf{p}^{i-1})g^{-1}_{m_i}=\textbf{p}^{i}$.

\smallskip

For any $X \subset \mathbb{N}^n$ we denote
\begin{equation*}
\downarrow X = \left\{\textbf{a}\in \mathbb{N}^n \colon \hbox{there xists~} \textbf{x}\in X \hbox{~such that~} \textbf{a}\leqslant \textbf{x} \right\}.
\end{equation*}

Put $A_0=B_0=\varnothing$.
For any $i\geq1$ choose elements ${\textbf{a}^i\in A\setminus \downarrow(A_{i-1} \cup B_{i-1})}$ and \break${\textbf{b}^i\in B\setminus \downarrow(A_{i-1} \cup B_{i-1})}$ and choose a path $p_i=(\textbf{p}^{1},\ldots, \textbf{p}^{k})$ from $\textbf{a}^i$ to $\textbf{b}^i$ with property that all $\textbf{p}^j \notin A_{i-1}\cup B_{i-1}$. By choosing the path $p_i$, there exists point $\textbf{p}^j$ of this path such that $\textbf{p}^j\in A$ and $\textbf{p}^{j+1}\in B$, so define the sets $A_i=A_{i-1}\cup\{\textbf{p}^j\}$ and $B_i=B_{i-1}\cup\{\textbf{p}^{j+1}\}$.

Next, we define $\tilde{C} =\displaystyle \bigcup\limits_{i=1}^{\infty} A_i$. We remark that for any $\textbf{a}\in \tilde{C}$ there exist $k\in\{1,\ldots,n\}$ and $s\in\{1,-1\}$ such that $(\textbf{a})g^{s}_{k}\in \displaystyle \bigcup\limits_{i=1}^{\infty} B_i\subset B$, denote this numbers by $k_\textbf{a}$ and $s_\textbf{a}$, respectively. Since the set $\tilde{C}$ is infinite there exist infinite subset $C\subset\tilde{C}$ such that for any $\textbf{c}\in \tilde{C}$ the numbers $k_\textbf{c}$ and $s_\textbf{c}$ is the same.
\end{proof}

\begin{lemma}\label{lemma-1.5}
Let $({\mathscr{I\!\!P\!F}(\mathbb{N}^n)}^0,\tau)$ be a locally compact non-discrete semitopological semigroup. Then for any neighborhood $U(0)$ of the zero $0$ and for any permutation $\sigma \in \mathscr{S}_n$ there exist $\textbf{a}\in \mathbb{N}^n$ such that the set $L_\sigma^\textbf{a} \setminus U(0)$ is finite.
\end{lemma}

\begin{proof}[\textsl{Proof}]
Fix any neighborhood $U(0)$ of the zero $0$ and any permutation $\sigma \in \mathscr{S}_n$. Lemma~\ref{lemma-1.2} implies that there exist $\textbf{a}\in \mathbb{N}^n$ such that the set $L_\sigma^\textbf{a} \cap U(0)$ is infinite.

\smallskip

The set $L_\sigma^\textbf{a}$ is a disjoint union: $L_\sigma^\textbf{a} = (L_\sigma^\textbf{a}\cap U(0)) \sqcup (L_\sigma^\textbf{a}\setminus U(0))$. The statements of the lemma would be proved when the set $L_\sigma^\textbf{a}\setminus U(0)$ is finite, and hence we assuming the opposite that the set $L_\sigma^\textbf{a}\setminus U(0)$ is infinite.

\smallskip

We consider the bijection $f_{\sigma}^{a}\colon L_\sigma^\textbf{a} \rightarrow \mathbb{N}^n$ defined by the formula
\begin{equation*}
(\sigma, [\textbf{a},\textbf{x}])f_{\sigma}^{a}=\textbf{x}.
\end{equation*}
Lemma~\ref{lemma-1.3} implies that in the set $L_\sigma^\textbf{a}\cap U(0)$ there exists an infinite subset $C$ and an integer number $k\in\{1,\ldots,n\}$ such that at least one of two sets $(C)(f_{\sigma}^{a}\circ g_k)$ and $(C)(f_{\sigma}^{a}\circ g^{-1}_k)$ is a subset of $(L_\sigma^\textbf{a}\setminus U(0))f_{\sigma}^{a}$.

\smallskip

We remark that the composition $f_{\sigma}^{a}\circ g_k\circ (f_{\sigma}^{a})^{-1}$ coincide with the restriction of right translation
$\rho_{(1,[\textbf{0},\textbf{1}_k])}$ to the set $L_\sigma^\textbf{a}$, i.e.,
$$f_{\sigma}^{a}\circ g_k\circ (f_{\sigma}^{a})^{-1}=\rho_{(1,[\textbf{1},\textbf{2}_k])}{\big|}_{L_\sigma^\textbf{a}}$$
and similarly
$$f_{\sigma}^{a}\circ g_k^{-1}\circ (f_{\sigma}^{a})^{-1}=\rho_{(1,[\textbf{2}_k,\textbf{1}])}{\big|}_{L_\sigma^\textbf{a}\setminus \{(\sigma,[\textbf{a},\textbf{x}])  \colon \textbf{x}\in\mathbb{N}^n, \; x_k=2  \}}$$

Lemma~\ref{lemma-1.1} and the separate continuity of the semigroup operation in $({\mathscr{I\!\!P\!F}(\mathbb{N}^n)}^0,\tau)$ imply that there exists an open compact neighbourhood $V(0)$ of the zero $0$ in $({\mathscr{I\!\!P\!F}(\mathbb{N}^n)}^0,\tau)$ such that \break $V~\cdot~(1, [\textbf{1}, \textbf{2}_k])~\subset~U(0)$ and $V~\cdot~(1, [\textbf{2}_k, \textbf{1}])~\subset~ U(0)$.

In any case we have that the set $C$ is a subset of $U(0)\setminus V(0)$. Indeed:
\begin{itemize}
  \item[$(i)$] if $(C)(f_{\sigma}^{a}\circ g_k)$ is subset of $(L_\sigma^\textbf{a}\setminus U(0))f_{\sigma}^{a}$ then we have that
\end{itemize}
\begin{equation*}
\begin{split}
  C~\cdot~(1, [\textbf{1}, \textbf{2}_k]) & =(C)\rho_{(1,[\textbf{1},\textbf{2}_k])}= \\
    & =(C)\rho_{(1,[\textbf{1},\textbf{2}_k])}{\big|}_{L_\sigma^\textbf{a}}=\\
    & =(C)(f_{\sigma}^{a}\circ g_k\circ (f_{\sigma}^{a})^{-1}) \subset\\
    & \subset L_\sigma^\textbf{a}\setminus U(0);
\end{split}
\end{equation*}

\begin{itemize}
  \item[$(ii)$] if $(C)(f_{\sigma}^{a}\circ g^{-1}_k)$ is subset of $(L_\sigma^\textbf{a}\setminus U(0))f$ then we have that
\end{itemize}
\begin{equation*}
\begin{split}
  C~\cdot~(1, [\textbf{2}_k, \textbf{1}]) & =(C)\rho_{(1,[\textbf{2}_k, \textbf{1}])}= \\
    & =(C)\rho_{(1,[\textbf{2}_k,\textbf{1}])}{\big|}_{L_\sigma^\textbf{a}\setminus \{(\sigma,[\textbf{a},\textbf{x}]) \hbox{ } | \hbox{ } \textbf{x}\in\mathbb{N}^n, \hbox{ } x_k=2  \}}=\\
    & = (C)(f_{\sigma}^{a}\circ g^{-1}_k\circ (f_{\sigma}^{a})^{-1}) \subset \\
    & \subset L_\sigma^\textbf{a}\setminus U(0);
\end{split}
\end{equation*}
and since $C$ is an infinite set this contradicts Lemma~\ref{lemma-1.1}(2).
\end{proof}

\begin{lemma}\label{lemma-1.6}
Let $({\mathscr{I\!\!P\!F}(\mathbb{N}^n)}^0,\tau)$ be a locally compact non-discrete semitopological semigroup. Then for any neighborhood $U(0)$ of the zero $0$, any permutation $\sigma \in \mathscr{S}_n$ and any element $\textbf{a}\in \mathbb{N}^n$ the set $L_\sigma^\textbf{a} \setminus U(0)$ is finite.
\end{lemma}

\begin{proof}[\textsl{Proof}]
Consider any neighborhood $U(0)$ of the zero $0$ and any permutation $\sigma \in \mathscr{S}_n$. Lemma~\ref{lemma-1.5} implies that there exists $\textbf{b}\in \mathbb{N}^n$ such that the set $L_\sigma^{\textbf{b}} \setminus U(0)$ is finite. Fix any $\textbf{a}\in \mathbb{N}^n\setminus \{\textbf{b}\}$. Define elements $\textbf{q},\textbf{p}\in \mathbb{N}^n$ in the following way: for any $i\in \{1,\ldotp,n\}$ put
$$q_i=1, \quad p_i=b_i-a_i, \qquad \hbox{if} \quad  b_i\geq a_i;$$
$$p_i=1, \quad q_i=a_i-b_i, \qquad \hbox{if} \quad b_i<a_i.$$
We remark that $\textbf{q}-\textbf{p}=\textbf{a}-\textbf{b}$ and $\max\{\textbf{p},\textbf{b}\} = \textbf{b}$. Then, the restriction of the left translation $\lambda_{(1,[(\textbf{q})\sigma^{-1},(\textbf{p})\sigma^{-1}])}$ on the set $L_\sigma^\textbf{b}$ is a bijection between $L_\sigma^\textbf{b}$ and $L_\sigma^\textbf{a}$:
for any $(\sigma,[\textbf{b},\textbf{x}])\in L_\sigma^\textbf{b}$ we have that
\begin{equation*}
\begin{split}
  (\sigma,[\textbf{b},\textbf{x}])\lambda_{(1,[(\textbf{q})\sigma^{-1},(\textbf{p})\sigma^{-1}])} & =(1,[(\textbf{q})\sigma^{-1},(\textbf{p})\sigma^{-1}])\cdot(\sigma,[\textbf{b},\textbf{x}])= \\
    & =(\sigma, [\textbf{q},\textbf{p}]\ast[\textbf{b},\textbf{x}])=\\
    & =(\sigma, [\max\{\textbf{p},\textbf{b}\}-\textbf{p}+\textbf{q},\max\{\textbf{p},\textbf{b}\}-\textbf{b}+\textbf{x}])=\\
    & =(\sigma, [\textbf{b}-\textbf{p}+\textbf{q},\textbf{x}])=(\sigma, [\textbf{a},\textbf{x}]).
\end{split}
\end{equation*}
Lemma~\ref{lemma-1.1} and the separate continuity of the semigroup operation in $({\mathscr{I\!\!P\!F}(\mathbb{N}^n)}^0,\tau)$ imply that there exist an open compact neighbourhood $V(0)$ of the zero $0$ in $({\mathscr{I\!\!P\!F}(\mathbb{N}^n)}^0,\tau)$ \break such that $(1,[(\textbf{q})\sigma^{-1},(\textbf{p})\sigma^{-1}])~\cdot~V(0)~\subset~U(0)$.
Since
\begin{equation*}
\begin{split}
  L_\sigma^\textbf{a}\setminus U(0) & \subseteq L_\sigma^\textbf{a}\setminus (1,[(\textbf{q})\sigma^{-1},(\textbf{p})\sigma^{-1}])~\cdot~V(0)= \\
    & =L_\sigma^\textbf{a} \setminus (V(0))\lambda_{(1,[(\textbf{q})\sigma^{-1},(\textbf{p})\sigma^{-1}])}=\\
    & =(L_\sigma^\textbf{b}\setminus V(0))\lambda_{(1,[(\textbf{q})\sigma^{-1},(\textbf{p})\sigma^{-1}])},
\end{split}
\end{equation*}
$L_\sigma^\textbf{a}\setminus U(0)$ is finite, because the set $L_\sigma^\textbf{b}\setminus V(0)$ is finite.
\end{proof}

\begin{lemma}\label{lemma-1.7}
Let $({\mathscr{I\!\!P\!F}(\mathbb{N}^n)}^0,\tau)$ be a locally compact non-discrete semitopological semigroup. Then for any neighborhood $U(0)$ of the zero $0$ and for any permutation $\sigma \in \mathscr{S}_n$ there exist only finite number of elements $\textbf{a}\in \mathbb{N}^n$ such that the set $L_\sigma^\textbf{a} \setminus U(0)$ is non empty, i.e. the set $\left\{\textbf{a}\in \mathbb{N}^n \colon L_\sigma^\textbf{a} \setminus U(0) \neq \varnothing\right\}$ is finite.
\end{lemma}

\begin{proof}[\textsl{Proof}]
Suppose to the contrary that there exist neighborhood $U(0)$ of the zero $0$ and permutation $\sigma \in \mathscr{S}_n$ such that the set $M=\left\{\textbf{b}\in \mathbb{N}^n \colon L_\sigma^\textbf{b} \setminus U(0) \neq \varnothing\right\}$ is infinite. Since for any $\textbf{b}\in M$, by Lemma~\ref{lemma-1.6}, the set $L_\sigma^\textbf{b}\setminus U(0)$ is finite, there exist an element 
$\textbf{x}_\textbf{b}\in \mathbb{N}^n$
and a positive integer $k_\textbf{b}\in\{1,\ldots,n\}$ such that $(\sigma,[\textbf{b},\textbf{x}_\textbf{b}])\notin L_\sigma^\textbf{b}\setminus U(0)$ and
${(\sigma,[\textbf{b},\textbf{x}_\textbf{b}-(0,\ldots,\underbrace{1}_{k_\textbf{b}}, \ldots, 0)])\in L_\sigma^\textbf{b}\setminus U(0)}$. This defines the maps:
$$\gamma\colon M \rightarrow \mathbb{N}^n, \; \textbf{b} \mapsto \textbf{x}_\textbf{b},$$
$$\phi\colon M \rightarrow \{1,\ldots, n\}, \; \textbf{b} \mapsto k_\textbf{b}.$$
Since $M$ is infinite and $(M)\phi$ finite, there exist an infinite subset $M^\prime\subset M$ and positive integer ${k_{M^\prime}\in\{1,\ldots,n\}}$ such that for any two elements $\textbf{u},\textbf{v}\in M^\prime$ the following equality
$$(\textbf{u})\phi=(\textbf{v})\phi=k_{M^\prime}$$
holds.

Lemma~\ref{lemma-1.1} and the separate continuity of the semigroup operation in $({\mathscr{I\!\!P\!F}(\mathbb{N}^n)}^0,\tau)$ imply that there exists an open compact neighbourhood $V(0)$ of the zero $0$ in $({\mathscr{I\!\!P\!F}(\mathbb{N}^n)}^0,\tau)$ such that \break ${V(0)\cdot(1,[\textbf{2}_{k_{M^\prime}},1])\subset U(0)}$.

Put $P=\left\{(\sigma,[\textbf{b},(\textbf{b})\gamma]) \colon \textbf{b}\in M^\prime\right\}$. Then the choice of $M^\prime$ implies $P\subset U(0)\setminus V(0)$, which contradicts Lemma~\ref{lemma-1.1}(2), because the set $M^\prime$ is infinite.
\end{proof}


\begin{corollary}\label{lemma-1.8}
Let $({\mathscr{I\!\!P\!F}(\mathbb{N}^n)}^0,\tau)$ be a locally compact non-discrete semitopological semigroup. Then for any neighborhood $U(0)$ of the zero $0$ the set ${\mathscr{I\!\!P\!F}(\mathbb{N}^n)}^0 \setminus U(0)$ is finite.
\end{corollary}

\begin{proof}[\textsl{Proof}]
Since
\begin{equation*}
\mathscr{I\!\!P\!F}(\mathbb{N}^n) = \bigsqcup\limits_{\sigma\in \mathscr{S}_n} \{\sigma\}\times(\bigsqcup\limits_{\textbf{a}\in \mathbb{N}^n} L_\sigma^\textbf{a}),
\end{equation*}
Lemma~\ref{lemma-1.7}, implies that the set
\begin{equation*}
\mathscr{I\!\!P\!F}(\mathbb{N}^n) \setminus U(0) =
{\bigsqcup\limits_{\sigma\in \mathscr{S}_n} \{\sigma\}\times\bigsqcup\limits_{\textbf{a}\in \mathbb{N}^n} L_\sigma^\textbf{a} \setminus U(0)} = {\bigsqcup\limits_{\sigma\in \mathscr{S}_n} \{\sigma\}\times\bigsqcup\limits_{\textbf{a}\in \mathbb{N}^n} L_\sigma^\textbf{a}\setminus U(0)}
\end{equation*}
is finite.
\end{proof}

\begin{example}\label{example-1.9}
We define a topology $\tau_{Ac}$ on the semigroup ${\mathscr{I\!\!P\!F}(\mathbb{N}^n)}^0$ in the followinh way:
\begin{itemize}
  \item[$(i)$] every element of the semigroup ${\mathscr{I\!\!P\!F}(\mathbb{N}^n)}$ is an isolated point in the space $({\mathscr{I\!\!P\!F}(\mathbb{N}^n)}^0,\tau_{Ac})$;
  \item[$(ii)$] the family $\mathcal{B}_{Ac}(0) = \left\{U \subset {\mathscr{I\!\!P\!F}(\mathbb{N}^n)}^0 : U \ni 0 \hbox{ and }  {\mathscr{I\!\!P\!F}(\mathbb{N}^n)} \setminus U \hbox{ is finite}\right\}$ de\-ter\-mines a base of the topology $\tau_{Ac}$ at zero $0 \in {\mathscr{I\!\!P\!F}(\mathbb{N}^n)}^0$
\end{itemize}
i.e., $\tau_{Ac}$ is the topology of the Alexandroff one-point compactification of the discrete space ${\mathscr{I\!\!P\!F}(\mathbb{N}^n)}$ with the remainder ${0}$. The semigroup operation in $({\mathscr{I\!\!P\!F}(\mathbb{N}^n)}^0,\tau)$ is separately continuous, because all elements of the semigroup ${\mathscr{I\!\!P\!F}(\mathbb{N}^n)}$ are isolated points in the space $({\mathscr{I\!\!P\!F}(\mathbb{N}^n)}^0,\tau)$ and any first order equation in $\mathscr{I\!\!P\!F}(\mathbb{N}^n$ has finitely many solutions (see Proposition~2.26 in \cite{Gutik-Mokrytskyi-2019}).
\end{example}

\begin{remark}\label{remark-1.3}
In \cite{Gutik-Mokrytskyi-2019} showed that the discrete topology $\tau_d$ is a unique topology on the semigroup ${\mathscr{I\!\!P\!F}(\mathbb{N}^n)}$ such that this is a semitopological semigroup. So $\tau_{Ac}$ is the unique compact topology on this semigroup such that $({\mathscr{I\!\!P\!F}(\mathbb{N}^n)}^0,\tau_{Ac})$ is a compact semitopological semigroup.
\end{remark}

Lemma~\ref{lemma-1.8} and Remark~\ref{remark-1.3} implies the following dichotomy for a locally compact semitopological semigroup ${\mathscr{I\!\!P\!F}(\mathbb{N}^n)}^0$.

\begin{theorem}\label{theorem-1}
If ${\mathscr{I\!\!P\!F}(\mathbb{N}^n)}^0$ is a Hausdorff locally compact semitopological semigroup, then either ${\mathscr{I\!\!P\!F}(\mathbb{N}^n)}^0$ is discrete or ${\mathscr{I\!\!P\!F}(\mathbb{N}^n)}^0$ is topologically isomorphic to $({\mathscr{I\!\!P\!F}(\mathbb{N}^n)}^0,\tau_{Ac})$.
\end{theorem}

By Corollary~3.3 of \cite{Gutik-Mokrytskyi-2019} the semigroup $\mathscr{I\!\!P\!F}(\mathbb{N}^n)$ does not embed into a compact Hausdorff topological semigroup. Hence Theorem~\ref{theorem-1} implies the following corollary:

\begin{corollary}\label{corollary-1}
If ${\mathscr{I\!\!P\!F}(\mathbb{N}^n)}^0$ is a Hausdorff locally compact topological semigroup  then the space ${\mathscr{I\!\!P\!F}(\mathbb{N}^n)}^0$ is discrete.
\end{corollary}

\section*{\textbf{Acknowledgements}}

The authors are grateful to the referee and his PhD advisor Oleg Gutik for several useful comments and suggestions.

\end{document}